\documentclass[12pt,a4paper,intlimits,sumlimits]{amsart}

\usepackage[usenames,dvipsnames,svgnames,table,rgb]{xcolor}
\usepackage{graphicx,txfonts,latexsym,fancyhdr,calc}
\usepackage{amsfonts,amsmath,amsthm,amssymb}
\usepackage[
allcolors=blue,allbordercolors=blue,pdfborderstyle={/S/U/W 1}]{hyperref}
\usepackage[ocgcolorlinks]{ocgx2}
\usepackage{epigraph}
\usepackage{tcolorbox}
\usepackage{asymptote}
\usepackage{parskip}
\usepackage{pgfgantt}
\usepackage{tcolorbox}
\usepackage{xfrac}

\ifdefined\SMART 
\usepackage[paperwidth=12cm,paperheight=24cm,left=5mm,right=5mm,top=10mm,bottom=5mm]{geometry}
\else
\usepackage[left=2cm,right=2cm,bottom=2cm,top=2cm]{geometry} 
\fi

\usepackage[normalem]{ulem}

\usepackage{txfonts,pxfonts,tikz} 
\usepackage{tgschola}
\usepackage[T1]{fontenc}

\theoremstyle{plain}
\newtheorem{theorem}{Theorem}

\newtheorem{lemma}{Lemma}
\newtheorem{corollary}{Corollary}
\newtheorem{proposition}{Proposition}

\theoremstyle{definition}

\newtheorem{remark}{Remark}

\newtheorem{example}{Example}

\newtheorem{case[theorem]}{Case}

\newcommand{\beql}[1]{\begin{equation}\label{#1}}
\newcommand{\eeq}{\end{equation}}
\newcounter{rem}
\newcounter{step}
\newcounter{mysec}
\newcounter{mysubsec}[mysec]
\title{Spectrality of Weighted Measures on Two Line Segments}
\author{Sha Wu}
\address{\href{https://mi.gxu.edu.cn/}{School of Mathematics, Guangxi University}, 530004  Nanning, China}
\email{shawmath@gxu.edu.cn}
\thanks{ }

\begin{document}
\makeatletter
\@namedef{subjclassname@2020}{\textup{2020} Mathematics Subject Classification}
\makeatother
\date{23 September 2026}

 \begin{abstract}
We study the spectrality of measures with positive integrable
densities supported on two line segments in \(\mathbb R^d\).
We prove that, when the two segments are non-overlapping,
spectrality forces the density on each segment to be constant
almost everywhere.  When the two segments are overlapping, spectrality forces the
total density to be constant almost everywhere on their union.
We then study the resulting measures with positive constant
densities according to the geometric of the
segments. For two non-coplanar segments, every such measure
admits a spectrum contained in a straight line. For two segments lying on distinct parallel lines, the measure is spectral if and only if the two
densities are equal. For non-parallel coplanar segments, a suitable
invertible linear transformation converts the  measure into an unweighted arc-length measure. When these segments are viewed in their affine plane, every spectrum of a spectral measure is contained in a straight
line.  Finally, we give examples showing how the densities determine the
directions of line spectra and construct explicit spectra
for weighted measures.
\end{abstract}
 \maketitle
\tableofcontents

\section{Introduction}
\subsection{Spectral measures} The classical Fourier series on the unit interval provides one of the
simplest and most fundamental examples of an orthonormal exponential
basis. In fact,
 $\{e^{2\pi i n x}:n\in\mathbb Z\}$
forms an orthonormal basis for \(L^2([0,1])\). This idea can be extended to more general measures, leading naturally to the notion of spectral measures. 
Let \(\mu\) be a Borel probability measure on \(\mathbb R^d\). 
The measure \(\mu\) is called a \emph{spectral measure} if there exists
a countable set \(\Lambda\subset\mathbb R^d\) such that
 $ \{e_{\lambda}:=e^{2\pi i\langle\lambda,x\rangle}\}_{\lambda\in\Lambda}$
forms an orthonormal basis for \(L^2(\mu)\). In this case,
\(\Lambda\) is called a \emph{spectrum} of \(\mu\), and the pair
\((\mu,\Lambda)\) is called a \emph{ spectral pair}.  In particular, if the  Lebesgue measure restricted on the set $\Omega$ is  a spectral measure,  then we say $\Omega$ is a \emph{ spectral set}.
The orthogonality of exponential functions can be determined by the zeros of the Fourier transform of  $\mu$. Indeed, for any distinct \(\lambda,\lambda'\in\Lambda\),
\[
\begin{aligned}
    \langle e_\lambda,e_{\lambda'}\rangle_{L^2(\mu)}
     =
    \int_{\mathbb R^d}
    e^{2\pi i\langle\lambda-\lambda',x\rangle}\,d\mu(x) =
    \widehat{\mu}(\lambda'-\lambda).
\end{aligned}
\]
Consequently, \(\Lambda\) is an orthogonal set for \(\mu\) if and
only if
   $ (\Lambda-\Lambda)\setminus\{0\}
    \subseteq\mathcal Z(\widehat{\mu})$,
where
\[
    \mathcal Z(\widehat{\mu})
    =
    \{\xi\in\mathbb R^d:\widehat{\mu}(\xi)=0\}.
\]
And the completeness of
\(E(\Lambda)\) can equivalently be expressed by
\[
    \sum_{\lambda\in\Lambda}
    \left|
        \widehat{\mu}(\xi+\lambda)
    \right|^2
    =
    1 \]
 for  $ \xi\in\mathbb R^d$.
Thus, the study of spectral measures is closely connected with the
geometry of the zero set of the Fourier transform and the distribution
of discrete frequency sets.

The notion of spectral measures extends the classical framework of spectral sets from Lebesgue measure on subsets of Euclidean space to more general Borel measures, including singular and fractal measures. A decisive development was the work of Jorgensen and
Pedersen, who constructed singular continuous self-similar measures, namely $\frac{1}{4}$-Cantor measures admitting orthonormal bases of exponential functions
\cite{JorgensenPedersen1998}. Since then, the spectrality of various classes of measures has been extensively studied, including self-similar measures, self-affine measures, infinite convolution measures and Moran measures (see \cite{AHH19,AH14,AHL22,DHL2013,DHL2014,DL21,DL23,ducasse2025spectral,DHL2019,FFS16,FY97,FHL15,FHW18,kolountzakis2025non}).

A natural generalization of an orthonormal basis is a Fourier frame.
Let \(\mu\) be a finite Borel measure on \(\mathbb R^d\). A countable
set \(\Lambda\subset\mathbb R^d\) is called a \emph{  frame
spectrum} for \(L^2(\mu)\) if there exist constants \(0<A\leq B<\infty\)
such that
\[
    A\|f\|_{L^2(\mu)}^2
    \leq
    \sum_{\lambda\in\Lambda}
    \left|
        \int_{\mathbb R^d}
        f(x)e^{-2\pi i\langle\lambda,x\rangle}\,d\mu(x)
    \right|^2
    \leq
    B\|f\|_{L^2(\mu)}^2
\]
for every \(f\in L^2(\mu)\).  In this case, \(E(\Lambda)=\{e_\lambda:\lambda\in\Lambda\}\) is called
a \emph{Fourier frame} for \(L^2(\mu)\), and \(\mu\) is called an
\emph {frame spectral measure} with frame spectrum \(\Lambda\). 
If the two frame bounds can be chosen to be equal, that is, if there
exists a constant \(A>0\) such that
\[
    \sum_{\lambda\in\Lambda}
    \left|
        \int_{\mathbb R^d}
        f(x)e^{-2\pi i\langle\lambda,x\rangle}\,d\mu(x)
    \right|^2
    =
    A\|f\|_{L^2(\mu)}^2,
    \qquad f\in L^2(\mu),
\]
then \(E(\Lambda)\) is called a \emph{tight  frame}. When
\(A=1\), it is called a \emph{Parseval  frame}. Every
orthonormal basis is a Parseval frame, although a Parseval frame need
not be an orthonormal basis, since a frame may contain redundancy.
The uniformity results of Dutkay and Lai \cite{DL2014} show that the existence of a
Fourier frame imposes strong restrictions on the density of an
absolutely continuous measure. In particular, the existence of a tight Fourier frame of exponential functions forces the
density to be constant on its support. These results demonstrate that the existence of a Fourier frame is
closely related to the structure of the underlying measure and may
impose strong restrictions on its density and distribution.

\subsection{Translational tiling and the Fuglede conjecture}
Let \(\Omega\subset\mathbb R^d\) be a measurable set of
positive and finite Lebesgue measure. We say that \(\Omega\) tiles
\(\mathbb R^d\) by translations if there exists a discrete set
\(\mathcal T\subset\mathbb R^d\) such that
\[
    \sum_{\tau\in\mathcal T}
    \mathbf 1_\Omega(x-\tau)
    =
    1
\]
for almost every \(x\in\mathbb R^d\). Equivalently, the translates
\(\{\Omega+\tau:\tau\in\mathcal T\}\) cover \(\mathbb R^d\), up to a
set of measure zero, and any two distinct translates overlap only on
sets of measure zero. In this case, \(\Omega\) is called a
\emph{translational tile}, and \(\mathcal T\) is called  a \emph{translation set} for \(\Omega\).
Fuglede proposed in 1974 the following conjecture
\cite{fuglede1974operators}:
\[\bf{
    \Omega\text{ \bf is a spectral set}
    \quad\bf{\Longleftrightarrow}\quad
    \Omega\text{ \bf  is a translational tile} }
\]
Fuglede proved this equivalence when either the spectrum or the tiling
set is a lattice. In this case, if \(\Omega\) is a fundamental domain
of a full-rank lattice \(\Gamma\), then the dual lattice
\[
    \Gamma^*
    =
    \{\lambda\in\mathbb R^d:
      \langle\lambda,\gamma\rangle\in\mathbb Z
      \text{ for every }\gamma\in\Gamma\}
\]
is a spectrum of \(\Omega\). However, this equivalence does not hold in general.  Tao \cite{tao2004fuglede}
constructed a spectral set that does not translational tile in dimensions
\(d\geq5\). Subsequent constructions showed that both
directions of the conjecture fail in every dimension \(d\geq3\)  (see 
\cite{farkas2006onfuglede,farkas2006tiles,kolountzakis2006tiles,kolountzakis2006hadamard}).  More recently, counterexamples to both
directions in dimension two were announced in \cite{Zhang2026}.
The conjecture therefore remains open  in
dimension one.

Despite these counterexamples, the equivalence remains valid for
several important classes. For example, \L aba \cite{laba2001} proved that the Fuglede
conjecture holds for a union of two intervals in \(\mathbb R\). Lev and Matolcsi \cite{LevMatolcsi2022} established the conjecture for
convex bodies in all dimensions. Various
positive results are also known for finite abelian groups, lattice
tilings, product domains, and other geometrically restricted classes.

\subsection{Spectral measures supported on two line segments}
The classical example in dimension one is the unit interval
\(\Omega=[0,1]\), which is both a spectral set and a translational
tile. More precisely, the integer   \(\mathbb Z\) serves both as
a spectrum of \(\Omega\) and as a translation set that tiles
\(\mathbb R\) by copies of \(\Omega\). 
This fundamental example naturally raises the question of whether the connection between spectrality and translational tiling continues to hold for more general sets, particularly for unions of two intervals. An important result in this direction was obtained by \L aba
\cite{laba2001}, who proved that the Fuglede conjecture holds for a union of two disjoint intervals in \(\mathbb R\).
From a geometric point of view, a union of two intervals may be regarded as a pair of collinear line segments. This naturally suggests extending the problem to two line segments in the plane with arbitrary relative positions and directions. Unlike the collinear case, the corresponding arc-length measure is singular with respect to two-dimensional Lebesgue measure. The study of such measures was initiated by Lai, Liu, and Prince
\cite{lai2021spectral}, they considered the normalized
arc-length measure on two line segments obtained by placing the same
interval \([t,t+1]\) on the two coordinate axes. For \(t=0\), they
proved that this measure has the unique spectrum
\[
    \Lambda=\left\{\left(\frac n2,-\frac n2\right):n\in\mathbb Z\right\}
\]
up to translations. They also established non-spectrality when \(t=-\frac12+\frac{1}{2a}\) for an integer \(a>1\).
Ai, Lu, and Zhou \cite{ai2023spectrality}  further considered the normalized arc-length measure on two line segments obtained by placing the same interval \([t,t+1]\) on the two coordinate axes. For \(t\in\mathbb Q\setminus\{-\frac12\}\), they proved that this measure is spectral if and only if \(t\in\frac12\mathbb Z\).
  Kolountzakis and the author
\cite{kolountzakis2025spectrality} extended this investigation by establishing
non-spectrality for \(-\frac12<t<0\) and for irrational \(t\), and by
giving a necessary and sufficient condition for the existence of a
spectrum contained in a straight line. More recently, Kolountzakis, Shi, and the author \cite{KSW} proved that the arc-length measure supported on two arbitrary line segments in the plane is spectral if and only if it admits a spectrum contained in a straight line. Moreover, when the two segments are non-parallel, every spectrum is contained in a straight line. As observed by Kolountzakis and the author
\cite{kolountzakis2025spectrality},  arc-length supported on two line
segments can be spectral even when the densities on the
two segments are different constants.
This naturally leads to the following question: {\bf if the
densities on the two segments are positive integrable
functions, under what conditions is the corresponding
measure spectral?}

\subsection{ Main results} 
To address the question above, we consider measures supported on
two line segments with positive integrable densities. Let
\(I_1,I_2\subset\mathbb R^2\) be two non-degenerate line segments,
and let
\[
    f_j\in L^1(I_j,\mathcal H^1)
    \qquad \text{and} \qquad
    f_j>0
    \quad \mathcal H^1\text{-a.e. on }I_j  
\]
for $ j=1,2$. We define
\begin{equation}\label{eq:weighted-segment-measure}
    \mu
    =
    f_1\,\mathcal H^1\!\restriction_{I_1}
    +
    f_2\,\mathcal H^1\!\restriction_{I_2},
\end{equation}
where \(\mathcal H^1\) denotes the one-dimensional Hausdorff
measure, which agrees with arc-length on each segment, and
\(\mathcal H^1\!\restriction_{I_j}\) denotes its restriction
to \(I_j\). Thus, \(f_j\) is the density of the contribution
from \(I_j\) with respect to arc-length. We assume that
\[
    \int_{I_1} f_1\,d\mathcal H^1
    +
    \int_{I_2} f_2\,d\mathcal H^1
    =1,
\]
so that \(\mu\) is a Borel probability measure.
In expressions involving both densities, each \(f_j\) is
extended by zero outside \(I_j\).

Our first result gives necessary conditions on the densities
for \(\mu\) to be spectral measure. If the two segments intersect in a set of measure zero, each density must be constant almost everywhere. If they overlap in a set of positive measure, their total density must be constant almost everywhere on the union.
\begin{theorem}\label{thm:density}
Let \(I_1,I_2\subset\mathbb R^d\), \(d\geq2\), be two
non-degenerate line segments. For \(j=1,2\), let
\[
    f_j\in L^1(I_j,\mathcal H^1) 
     \qquad \text{and} \qquad
    f_j>0
    \quad \mathcal H^1\text{-a.e. on }I_j,
\]
and let $\mu
    =
    f_1\mathcal H^1\!\restriction_{I_1}
    +
    f_2\mathcal H^1\!\restriction_{I_2}$
be a Borel probability measure. If \(\mu\) is a spectral
measure, then the following conclusions hold:
\begin{enumerate}
    \item[(i)]
    If \(\mathcal H^1(I_1\cap I_2)=0\), then there exist
    constants \(c_1,c_2>0\) such that
    \[
        f_j=c_j
        \quad \mathcal H^1\text{-a.e. on }I_j 
        \quad \text{for} \ j=1,2.
    \]
    Equivalently,
    $\mu
        =
        c_1\mathcal H^1\!\restriction_{I_1}
        +
        c_2\mathcal H^1\!\restriction_{I_2}.$
    \item[(ii)]
    If \(\mathcal H^1(I_1\cap I_2)>0\), then
    \[
        f_1\mathbf 1_{I_1}+f_2\mathbf 1_{I_2}
        =
        \frac{1}{\mathcal H^1(I_1\cup I_2)}
        \quad
        \mathcal H^1\text{-a.e. on }I_1\cup I_2.
    \]
    Equivalently,
        $\mu
        =
        \frac{1}{\mathcal H^1(I_1\cup I_2)}
        \mathcal H^1\!\restriction_{I_1\cup I_2}$, where
 \(\mathbf 1_{I_j}\) denotes the indicator function
    of \(I_j\).
\end{enumerate}
\end{theorem}

By above Theorem, when \(\mathcal H^1(I_1\cap I_2)=0\), the
spectrality problem of the measure $\mu$ by defined \eqref{eq:weighted-segment-measure}  reduces to the study of measures of the form
\[
    \mu
    =
    c_1\mathcal H^1\!\restriction_{I_1}
    +
    c_2\mathcal H^1\!\restriction_{I_2},
    \qquad c_1,c_2>0.
\]
Any two line segments in \(\mathbb R^d\) are contained in an affine
subspace of dimension at most three. Thus, they are either coplanar or
non-coplanar, and the latter case is intrinsically three-dimensional.
We first consider the non-coplanar case. The following theorem shows
that every measure of the above form supported on two non-coplanar line
segments is spectral and admits a spectrum contained in a straight line.
\begin{theorem}\label{thm:skew-segments}
Let \(I_1,I_2\subset\mathbb R^d (d\geq3)\)  be two
 non-coplanar line segments, and let
    $$\mu
    =
    c_1\mathcal H^1\!\restriction_{I_1}
    +
    c_2\mathcal H^1\!\restriction_{I_2},
    \qquad c_1,c_2>0$$
be a Borel probability measure. Then \(\mu\) is a spectral
measure and admits a spectrum contained in a straight line.
\end{theorem}
In the coplanar case, suppose that the common supporting plane is
\(p+V\), where
\[
    V=\operatorname{span}\{v_1,v_2\},
\]
with \(v_1,v_2\) orthonormal. Define the isometry
\[
    \Phi:p+V\longrightarrow\mathbb R^2,
    \qquad
    \Phi(p+s_1v_1+s_2v_2)=(s_1,s_2).
\]
Let
  $ \widetilde I_j=\Phi(I_j)$ for $ j=1,2$ 
and define
\[
    \widetilde\mu
    =c_1\mathcal H^1|_{\widetilde I_1}
    +c_2\mathcal H^1|_{\widetilde I_2}.
\]
Then \(\mu\) is spectral if and only if
\(\widetilde\mu\) is spectral. However, their spectra
may have different geometric structures, since the
spectrum of \(\mu\) lie in \(\mathbb R^d\), whereas
those of \(\widetilde\mu\) lie in \(\mathbb R^2\). Next, we explain the relation between their spectra. 
For \(x=p+s_1v_1+s_2v_2\),
\[
    e^{2\pi i\langle\lambda,x\rangle}
    =
    e^{2\pi i\langle\lambda,p\rangle}
    e^{2\pi i(s_1\langle\lambda,v_1\rangle
              +s_2\langle\lambda,v_2\rangle)}.
\]
Since the constant factor has modulus one, a spectrum
\(\Lambda\) of \(\mu\) corresponds to the spectrum
\[
    \widetilde\Lambda
    =
    \left\{
        \bigl(\langle\lambda,v_1\rangle,
              \langle\lambda,v_2\rangle\bigr):
        \lambda\in\Lambda
    \right\}
\]
of \(\widetilde\mu\). In other word, the planar spectrum \(\widetilde\Lambda\) is the orthogonal
projection of \(\Lambda\) onto \(V\), expressed in the
coordinates determined by \(v_1,v_2\).
Conversely, if \(\Gamma\subset\mathbb R^2\) is a
spectrum of \(\widetilde\mu\), then
\[
    \Lambda
    =
    \left\{
        \gamma_1v_1+\gamma_2v_2+\eta_\gamma:
        \gamma=(\gamma_1,\gamma_2)\in\Gamma
    \right\}
\]
is a spectrum of \(\mu\) for any choice of
\(\eta_\gamma\in V^\perp\).   Consequently,
when \(d>2\), a line spectrum of \(\widetilde\mu\)
may give a spectrum of \(\mu\) that is not contained
in any straight line.

Having reduced the coplanar case to \(\mathbb R^2\), we now
consider measures supported on two line segments in the plane.
We distinguish three cases according to the relative positions
of the supporting lines: collinear, parallel but distinct,
and non-parallel.

In the collinear case, the measure can be identified with an
absolutely continuous measure on \(\mathbb R\). The uniformity
result of Dutkay and Lai \cite{DL2014}  implies that its density must be constant
almost everywhere on \(J_1\cup J_2\), i.e., \(c_1=c_2\). The problem then reduces to
the spectrality of a union of two intervals, which is
characterized by \L aba  \cite{laba2001}.

We next consider two segments lying on distinct parallel lines.
After a rotation and a translation, we may assume that these
lines are horizontal. The following theorem shows that the
corresponding measure is spectral if and only if the two
constant densities are equal.
\begin{theorem}
\label{prop:weighted-parallel-segments}
Let \(h_1\neq h_2\in\mathbb R\),  
\(a_1,a_2\in\mathbb R\) and \(T_1,T_2,c_1,c_2>0\). And let
\[
    \mu
    =
    c_1
    \left(
        \mathcal L\!\restriction_{[a_1,a_1+T_1]}
        \times\delta_{h_1}
    \right)
    +
    c_2
    \left(
        \mathcal L\!\restriction_{[a_2,a_2+T_2]}
        \times\delta_{h_2}
    \right),
\]
is a Borel probability measure, where \(\mathcal L\) denotes the one-dimensional Lebesgue measure. Then $\mu$ is a spectral measure if and only if  $c_1=c_2=\frac{1}{T_1+T_2}$.
 \end{theorem}
Finally, we consider two non-parallel line segments with
positive constant densities. These densities need not be equal
for the measure to be spectral. The following theorem shows
that, whenever the measure is spectral measure, every spectrum is
contained in a straight line.
\begin{theorem}  
\label{non-p}
Let \(I_1,I_2\subset\mathbb R^2\) be two non-parallel line segments 
and let 
\[
    \mu
    =
    c_1\mathcal H^1\!\restriction_{I_1}
    +
    c_2\mathcal H^1\!\restriction_{I_2},
    \qquad c_1,c_2>0 
\]
be a  Borel probability measure.
Then  $ \mu$ is a spectral measure if and only if  every spectrum of \(\mu\) is contained in a straight line. 
\end{theorem}
 
\section{Proof of the main theorems}

\subsection{Uniformity of densities on two line segments}
We begin by proving the restrictions that spectrality imposes
on the densities. When the two segments intersect in a set
of zero arc-length, we show that each density is constant
almost everywhere on its segment. When their intersection
has positive arc length, we show that the total density
is constant almost everywhere on their union.

\begin{proof}[Proof of Theorem~\ref{thm:density}]
Let \(\Lambda\subset\mathbb R^d\) be a spectrum of \(\mu\).
We first prove \textup{(i)}. For \(j=1,2\), write
\[
    I_j=\{p_j+tv_j:0\leq t\leq T_j\},
\]
where \(p_j\in\mathbb R^d\), \(v_j\in\mathbb R^d\) is a
unit direction vector, and \(T_j=\mathcal H^1(I_j)\).
Define  $ w_j(t)=f_j(p_j+tv_j)$ for $t\in[0,T_j]$
and  \(w_j\) by zero outside \([0,T_j]\).
Fix \(j\in\{1,2\}\). Since
\(\mathcal H^1(I_1\cap I_2)=0\), every function
  $ g\in L^2([0,T_j],w_j(t)\,dt)$ 
can be identified with a function \(G\in L^2(\mu)\)
supported on \(I_j\), by setting
\[
    G(p_j+tv_j)=g(t)
\]
on \(I_j\) and \(G=0\) on the other segment.  
For \(\lambda\in\Lambda\), write
 $ e_\lambda(x)=e^{2\pi i\langle\lambda,x\rangle}.$
 Then
\[
\begin{aligned}
    \sum_{\lambda\in\Lambda}
    \left|\langle G,e_\lambda\rangle_{L^2(\mu)}\right|^2
    &=
    \sum_{\lambda\in\Lambda}
    \left|
        \int_0^{T_j}
        g(t)e^{-2\pi i\langle\lambda,p_j+tv_j\rangle}
        w_j(t)\,dt
    \right|^2\\
    &=
    \sum_{\lambda\in\Lambda}
    \left|
        \int_0^{T_j}
        g(t)e^{-2\pi i t\langle\lambda,v_j\rangle}
        w_j(t)\,dt
    \right|^2.
\end{aligned}
\]
Moreover,
\[
    \|G\|_{L^2(\mu)}^2
    =
    \int_0^{T_j}|g(t)|^2w_j(t)\,dt.
\]
Since \(E(\Lambda)\) is an orthonormal basis for
\(L^2(\mu)\), Parseval's identity gives
\begin{equation}\label{eq:density-parseval-j}
    \sum_{\lambda\in\Lambda}
    \left|
        \int_0^{T_j}
        g(t)e^{-2\pi i t\langle\lambda,v_j\rangle}
        w_j(t)\,dt
    \right|^2
    =
    \int_0^{T_j}|g(t)|^2w_j(t)\,dt.
\end{equation}
This implies that the projected exponential family
$\{ e^{2\pi i t\langle\lambda,v_j\rangle}\}_{\lambda\in\Lambda}$
  is a Parseval frame for
\(L^2([0,T_j],w_j(t)\,dt)\). 
By \cite[Corollary~2.6]{DL2014}, there exists
a constant \(c_j>0\) such that
\[
    w_j(t)=c_j
    \quad\text{for a.e. }t\in[0,T_j].
\]
Equivalently,
\[
    f_j=c_j
    \quad \mathcal H^1\text{-a.e. on }I_j.
\]
Since this holds for \(j=1,2\), we conclude that
\[
    \mu
    =
    c_1\mathcal H^1\!\restriction_{I_1}
    +
    c_2\mathcal H^1\!\restriction_{I_2}.
\]
This proves \textup{(i)}.

We now prove \textup{(ii)}. Suppose that
\(\mathcal H^1(I_1\cap I_2)>0\). Then \(I_1\) and
\(I_2\) are collinear. There
exist \(p\in\mathbb R^d\), a unit vector
\(v\in\mathbb R^d\), and \(T>0\) such that
\[
    I_1\cup I_2=\{p+tv:0\leq t\leq T\},
\]
where \(T=\mathcal H^1(I_1\cup I_2)\).
For \(j=1,2\), let
\[
    J_j=\{t\in[0,T]:p+tv\in I_j\}.
\]
After extending \(f_j\) by zero outside \(I_j\), define
\[
    w(t)
    =
    f_1(p+tv)\mathbf 1_{J_1}(t)
    +
    f_2(p+tv)\mathbf 1_{J_2}(t),
    \qquad 0\leq t\leq T.
\]

For every \(g\in L^2([0,T],w(t)\,dt)\),  since \(E(\Lambda)\) is an orthonormal basis for \(L^2(\mu)\),
Parseval's identity gives 
\[
    \sum_{\lambda\in\Lambda}
    \left|
        \int_0^T
        g(t)e^{-2\pi i t\langle\lambda,v\rangle}
        w(t)\,dt
    \right|^2
    =
    \int_0^T|g(t)|^2w(t)\,dt.
\]
Hence
$ \{ e^{2\pi i t\langle\lambda,v\rangle}\}_{\lambda\in\Lambda}$,
  is a Parseval frame
for \(L^2([0,T],w(t)\,dt)\). We can again apply
\cite[Corollary~2.6]{DL2014} to conclude that
there exists a constant \(C>0\) such that
\[
    w(t)=C
    \quad\text{for a.e. }t\in[0,T].
\]

Since \(\mu\) is a probability measure,
\[
    1
    =
    \mu(\mathbb R^d)
    =
    \int_0^T w(t)\,dt
    =
    CT.
\]
Thus
\[
    C=\frac1T
    =
    \frac{1}{\mathcal H^1(I_1\cup I_2)}.
\]
Consequently,
\[
    f_1\mathbf 1_{I_1}+f_2\mathbf 1_{I_2}
    =
    \frac{1}{\mathcal H^1(I_1\cup I_2)}
    \quad
    \mathcal H^1\text{-a.e. on }I_1\cup I_2,
\]
and hence
\[
    \mu
    =
    \frac{1}{\mathcal H^1(I_1\cup I_2)}
    \mathcal H^1\!\restriction_{I_1\cup I_2}.
\]
This proves \textup{(ii)}.
\end{proof}
\begin{remark}\label{rem:overlapping-densities}
In part \emph{(ii)} of Theorem \ref{thm:density},
spectrality does not imply that \(f_1\) and \(f_2\) are individually
constant.  
\end{remark}
The following example shows that the densities \(f_1\) and \(f_2\)
need not be individually constant, as stated in Remark \ref{rem:overlapping-densities}.
\begin{example} 
Let $I_1=[0,2]\times\{0\}$ and $I_2=[1,3]\times\{0\}.$ 
 Define \(f_1\) on \([0,2]\) by
\[
    f_1(x)
    =
    \begin{cases}
        \dfrac13,
            &0\leq x<1,\\[2mm]
      \dfrac{x}{12},
            &1\leq x\leq2,
    \end{cases}
\]
and define \(f_2\) on \([1,3]\) by
\[
    f_2(x)
    =
    \begin{cases}
        \dfrac13- \dfrac{x}{12},
            &1\leq x\leq2,\\[2mm]
        \dfrac13,
            &2<x\leq3.
    \end{cases}
\]
Both \(f_1\) and \(f_2\) are positive   on their
respective segments, and neither of them is constant.
Then $I_1\cap I_2=[1,2]\times\{0\}$
and hence $ \mathcal H^1(I_1\cap I_2)=1>0.$
 Consider the measure
\[
    \mu
    =
    f_1\,\mathcal H^1\!\restriction_{I_1}
    +
    f_2\,\mathcal H^1\!\restriction_{I_2}.
\]
On the overlapping part \(I_1\cap I_2\), we have
\[ f_1(x)+f_2(x)
    =\dfrac{x}{12}+\left(\frac13- \dfrac{x}{12}\right)=\frac13.
\]
On the non-overlapping parts, the density is also \(1/3\).
Consequently,
 $\mu=
    \frac13
    \mathcal H^1\!\restriction_{[0,3]\times\{0\}}.$
Thus, \(\mu\) is the normalized arc-length measure on a line segment
of length \(3\), and hence it is a spectral measure. For example,
\[
    \Lambda
    =
    \left\{
        \left(\frac n3,0\right):
        n\in\mathbb Z
    \right\}
\]
is a spectrum for \(\mu\).
\
Therefore, spectrality of
    $f_1\,\mathcal H^1\!\restriction_{I_1}
    +
    f_2\,\mathcal H^1\!\restriction_{I_2}$
does not imply that \(f_1\) and \(f_2\) are individually constant
when the two segments overlap in a set of positive arc length.
In this situation, spectrality can only determine the combined
density $f_1\mathbf 1_{I_1}
    +
    f_2\mathbf 1_{I_2}$,
rather than the two densities separately.
\end{example}
\subsection{Weighted measures on two non-coplanar line segments}
Kolountzakis and the author
\cite[Theorem~1.2]{kolountzakis2025spectrality} studied probability
measures supported on finite unions of line segments in
\(\mathbb R^2\). They showed that if the orthogonal
projection onto a  straight line \(L\) is one-to-one almost everywhere,
then the measure has a spectrum contained in \(L\) if and
only if the projected measure is a spectral measure. The same argument extends to projections onto linear subspaces of arbitrary dimension, as stated in the following lemma.
\begin{lemma}
Let \(\mu\) be a Borel probability measure on \(\mathbb R^d\), and
let \(V\subset\mathbb R^d\) be a \(k\)-dimensional linear subspace
with an orthonormal basis \(\gamma_1,\ldots,\gamma_k\). Define
\[
    \pi_V(x)
    =
    \bigl(\langle x,\gamma_1\rangle,\ldots,
          \langle x,\gamma_k\rangle\bigr)
    \in\mathbb R^k.
\]
Suppose that \(\pi_V\) is one-to-one \(\mu\)-almost everywhere.
Then, for any countable set \(\Lambda\subset\mathbb R^k\),
\[
    U\Lambda
    =
    \left\{
        \sum_{j=1}^k \lambda_j\gamma_j:
        (\lambda_1,\ldots,\lambda_k)\in\Lambda
    \right\}
    \subset V
\]
is a spectrum of \(\mu\) if and only if \(\Lambda\) is a spectrum
of \(\pi_V\mu\) on \(\mathbb R^k\), where 
\[
     \pi_V \mu(E)
    =
    \mu\bigl(\pi_V^{-1}(E)\bigr)
    =
    \mu\bigl(\{x\in\mathbb R^d:\pi_V(x)\in E\}\bigr)
\]
for every Borel set \(E\subset\mathbb R^k\).
\end{lemma}

\begin{proof}
We first consider orthogonality. For any
\(\lambda,\lambda'\in\Lambda\),
\[
\begin{aligned}
\langle U\lambda-U\lambda',x\rangle
&=\left\langle
   \sum_{j=1}^k(\lambda_j-\lambda'_j)\gamma_j,x
  \right\rangle\\
&=\sum_{j=1}^k(\lambda_j-\lambda'_j)
       \langle\gamma_j,x\rangle\\
&=\langle\lambda-\lambda',\pi_V(x)\rangle.
\end{aligned}
\]  
Hence, by the definition of the projected measure,
\[
\begin{aligned}
    \int_{\mathbb R^d}
        e^{2\pi i\langle U\lambda-U\lambda',~x\rangle}
        \,d\mu(x)
    &=
    \int_{\mathbb R^d}
        e^{2\pi i\langle\lambda-\lambda',~\pi_V(x)\rangle}
        \,d\mu(x)\\
    &=
    \int_{\mathbb R^k}
        e^{2\pi i\langle\lambda-\lambda',~s\rangle}
        \,d(\pi_V\mu)(s).
\end{aligned}
\]
Thus \(E(U\Lambda)\) is orthogonal in \(L^2(\mu)\) if and only if
\(E(\Lambda)\) is orthogonal in \(L^2(\pi_V\mu)\).

We next prove the equivalence of completeness. Since \(\pi_V\)
is one-to-one \(\mu\)-almost everywhere, every
\(f\in L^2(\mu)\) can be written as \(f=g\circ\pi_V\)
\(\mu\)-almost everywhere for some
\(g\in L^2(\pi_V\mu)\). Conversely, every \(g\in L^2(\pi_V\mu)\) determines a function \(f\in L^2(\mu)\) through \(f=g\circ\pi_V\).  By the definition of
the projected measure, we have
\[
    \|f\|_{L^2(\mu)}
    =
    \|g\|_{L^2(\pi_V\mu)}.
\]
Moreover, for every \(\lambda\in\Lambda\),
\[
\begin{aligned}
    \int_{\mathbb R^d}
        f(x)e^{-2\pi i\langle U\lambda,x\rangle}
        \,d\mu(x)
    &=
    \int_{\mathbb R^d}
        g(\pi_V(x))
        e^{-2\pi i\langle\lambda,\pi_V(x)\rangle}
        \,d\mu(x)\\
    &=
    \int_{\mathbb R^k}
        g(s)e^{-2\pi i\langle\lambda,s\rangle}
        \,d(\pi_V\mu)(s).
\end{aligned}
\]
Thus \(f\) is orthogonal to \(E(U\Lambda)\)  if
and only if \(g\) is orthogonal to \(E(\Lambda)\).
Since their norms are equal, \(f=0\) if and only if
\(g=0\) in the respective \(L^2\)-spaces. Hence
\(E(U\Lambda)\)  is complete in \(L^2(\mu)\)
if and only if \(E(\Lambda)\) is complete in
\(L^2(\pi_V\mu)\).
 This proves the lemma.
\end{proof}
Applying the preceding lemma to measures supported on finite unions
of line segments gives the following criterion for line spectra.
\begin{corollary}\label{lem:projection-rd}
Let \(\mu\) be a Borel probability measure on
\(\mathbb R^d\) (\(d\geq2\)), whose support is a finite
union of line segments, and let \(\gamma\in\mathbb R^d\)
be a unit vector. Define
\[
    \pi_\gamma(x)=\langle x,\gamma\rangle,
\]
which identifies the orthogonal projection onto
\(L=\mathbb R\gamma\) with a map into \(\mathbb R\).
Suppose that \(\pi_\gamma\) is one-to-one
\(\mu\)-almost everywhere. Then \(\mu\) has a spectrum
\[
    \Lambda\gamma
    =
    \{\lambda\gamma:\lambda\in\Lambda\}\subset L
\]
if and only if the projected measure \(\pi_\gamma\mu\)
has spectrum \(\Lambda\subset\mathbb R\), where
\[
 \pi_\gamma\mu(E)
    =
    \mu (\{x\in\mathbb R^d:
        \langle x,\gamma\rangle\in E\} )
\]
for every Borel set \(E\subset\mathbb R\).    
\end{corollary}
\begin{proof}[Proof of Theorem~\ref{thm:skew-segments}]
For  $j=1,2$, we write
\[
    I_j=\{p_j+tv_j:0\leq t\leq T_j\}
\]
where \(p_j\in\mathbb R^d\), \(v_j\in\mathbb R^d\) is a
unit direction vector, and \(T_j=\mathcal H^1(I_j)\). Since  \(I_1,I_2 \) are non-coplanar line segments, the vectors
   $ v_1, v_2$ and $w:=p_2-p_1$
are linearly independent. Hence there exists \(u\in\mathbb R^d\)
such that
\[
    \langle u,v_1\rangle=c_1,\qquad
    \langle u,v_2\rangle=c_2,\qquad
    \langle u,p_2-p_1\rangle=c_1T_1.
\]
Let $\bar{u}=\frac{u}{|u|}$ and  consider the orthogonal projection $ \pi_{\bar{u}}(x)=\langle x,\bar{u}\rangle$ onto \(\mathbb Ru\). Since \(\mu\) is a probability measure, we have
  $  c_1T_1+c_2T_2=1$.
By simple calculations, we obtain  
\[
    \pi_{\bar{u}}(p_1+tv_1)=b+\frac{c_1t}{|u|} 
    \qquad \text{and}\qquad
    \pi_{\bar{u}}(p_2+tv_2)=b+\frac{c_1T_1+c_2t}{|u|} 
\]
for \(0\leq t\leq T_j\), where  $b=\langle p_1, \bar{u} \rangle$. Thus
\[
    \pi_{\bar{u}}(I_1)=\left[b,b+\frac{c_1T_1}{|u|}\right]
  \qquad \text{and}\qquad
    \pi_{\bar{u}}(I_2)=
    \left[b+\frac{c_1T_1}{|u|},b+\frac{1}{|u|}\right].
\]
 Therefore,
\(\pi_{\bar{u}}\) is one-to-one \(\mu\)-almost everywhere  and the two
image intervals meet only at one endpoint.
To compute the projected measure $ \pi_{\bar{u}}  \mu$, let \(\varphi\) be a
nonnegative Borel function on \(\mathbb R\). Then 
\[
\begin{aligned}
    \int_{\mathbb R}\varphi(s)\,
        d\pi_{\bar{u}}\mu(s)
    &=
    \int_{\mathbb R^d}\varphi(\pi_{\bar{u}}(x))\,d\mu(x)\\
    &=
    c_1\int_0^{T_1}
        \varphi\left(b+\frac{c_1t}{|u|}\right)\,dt 
     +
    c_2\int_0^{T_2}
        \varphi\left(b+\frac{c_1T_1+c_2t}{|u|}\right)\,dt \\
    &=
    |u|\int_b^{b+c_1T_1/|u|}\varphi(s)\,ds
    +
    |u|\int_{b+c_1T_1/|u|}^{b+1/|u|}\varphi(s)\,ds\\
    &=
    |u|\int_b^{b+1/|u|}\varphi(s)\,ds.
\end{aligned}
\]
Thus, $$ \pi_{\bar{u}} \mu= |u|\,\mathcal L\!\restriction_{[b,b+\frac{1}{|u|}]} $$
and  the measure $ \pi_{\bar{u}} \mu$ has spectrum \(|u|\mathbb Z\). Hence, the measure \(\mu\) has a spectrum $ \Lambda  =\{nu:n\in\mathbb Z\}$ contained in
 the straight line \(\mathbb Ru\) by Corollary  \ref{lem:projection-rd}.
\end{proof}
The above theorem guarantees the existence of a line
spectrum. However, it does not imply that every spectrum
is contained in a straight line, as the following example
shows.

\begin{example}\label{ex:skew-nonline-spectrum}
Let
\[
    I_1=\{(t,0,0):0\leq t\leq1\},
    \qquad
    I_2=\{(0,t,1):0\leq t\leq1\},
\]
and define
\[
    \mu
    =
    \frac12\mathcal H^1\!\restriction_{I_1}
    +
    \frac12\mathcal H^1\!\restriction_{I_2}.
\]
 Their unit direction vectors and
the displacement between their initial points are
\[
    v_1=(1,0,0),\qquad
    v_2=(0,1,0),\qquad
    p_2-p_1=(0,0,1).
\]
These vectors are linearly independent, and hence the
supporting lines are non-coplanar.

Take $ u= (\frac12,\frac12,\frac12 )$. With \(c_1=c_2=1/2\) and \(T_1=T_2=1\), we have
\[
    \langle u,v_1\rangle=\frac12,\qquad
    \langle u,v_2\rangle=\frac12,\qquad
    \langle u,p_2-p_1\rangle=\frac12=c_1T_1.
\]
Therefore, Theorem~\ref{thm:skew-segments} shows that
\[
    \Lambda_0=
    \left\{
        \lambda_n^0=
        \left(\frac n2,\frac n2,\frac n2\right):
        n\in\mathbb Z
    \right\}
\]
is a spectrum of \(\mu\).
We now define
\[
    \lambda_n
    =
    \lambda_n^0+(0,0,n^2)
    =
    \left(\frac n2,\frac n2,\frac n2+n^2\right) \]
with $n\in\mathbb Z$ and let
 $ \Lambda=\{\lambda_n:n\in\mathbb Z\}$.
 For any  \((x,y,z)\in I_1\cup I_2\),  we have 
\[
\begin{aligned}
    e^{2\pi i\langle\lambda_n,(x,y,z)\rangle}
    =
    e^{2\pi i\langle\lambda_n^0,(x,y,z)\rangle}
    e^{2\pi i n^2z} 
    =
    e^{2\pi i\langle\lambda_n^0,(x,y,z)\rangle}
\end{aligned}.
\]
Thus, \(E(\Lambda)\) has the same orthogonality and completeness
properties as \(E(\Lambda_0)\).  
This means 
\(E(\Lambda)\) is also an orthonormal basis for
\(L^2(\mu)\),  but it is not contained in any straight line.

\end{example}
\subsection{  Weighted measures on two  parallel line segments}
We now turn to measures supported on two line segments lying
on distinct parallel lines. By Theorem \ref{thm:density}, spectrality forces
the density on each segment to be constant almost everywhere.
We show below that these two constants must also be equal,
which completes the characterization of spectrality in this case.
\begin{proof}[Proof of Theorem \ref{prop:weighted-parallel-segments}]  
Sufficiency can be directly obtained from \cite[Theorem~1]{KSW}. Next, we only need to consider necessity. 
Let
\(\Lambda\subset\mathbb R^2\) be a spectrum of \(\mu\) and write
 $\lambda=(\xi_\lambda,\eta_\lambda)$ for any  $\lambda\in\Lambda$. 
Thus 
\[
    E(\Lambda)
    =
    \left\{
        e_\lambda(x,y)
        =
        e^{2\pi i(\xi_\lambda x+\eta_\lambda y)}
        :
        \lambda\in\Lambda
    \right\}
\]
is an orthonormal basis for \(L^2(\mu)\).
Fix \(j\in\{1,2\}\), and let
 $ f\in L^2([a_j,a_j+T_j])$.
Define \(F_j\in L^2(\mu)\) by letting \(F_j\) agree with \(f\) on
the \(j\)-th segment and vanish on the other segment. 
We have
\[
    \|F_j\|_{L^2(\mu)}^2
    =
    c_j
    \int_{a_j}^{a_j+T_j}|f(x)|^2\,dx 
\]
and  for every \(\lambda\in\Lambda\), we have
\[
 \sum_{\lambda\in\Lambda}
    \left|
        \langle F_j,e_\lambda\rangle_{L^2(\mu)}
    \right|^2=
c_j^2\sum_{\lambda\in\Lambda}
    \left|
\int_{a_j}^{a_j+T_j}
f(x)e^{-2\pi i(\xi_\lambda x+\eta_\lambda h_j)}
\,dx   \right|^2= 
c_j^2\sum_{\lambda\in\Lambda}
    \left|  
\int_{a_j}^{a_j+T_j}
f(x)e^{-2\pi i\xi_\lambda x}\,dx \right|^2.
\]
Since \(E(\Lambda)\) is an orthonormal basis, Parseval's identity
gives
\begin{equation}\label{eq:parallel-frame-original}
    \sum_{\lambda\in\Lambda}
    \left|
        \int_{a_j}^{a_j+T_j}
        f(x)e^{-2\pi i\xi_\lambda x}\,dx
    \right|^2
    =
    \frac{1}{c_j}
    \int_{a_j}^{a_j+T_j}|f(x)|^2\,dx.
\end{equation}

Notice that the sum is indexed by \(\lambda\in\Lambda\), rather than
by the set of distinct first coordinates. Thus, if two different
elements of \(\Lambda\) have the same first coordinate, the
corresponding exponential occurs with the appropriate multiplicity.
Equation~\eqref{eq:parallel-frame-original} says precisely that the
sequence $ \{
        e^{2\pi i\xi_\lambda x}
     \}_{\lambda\in\Lambda}$
is a tight frame for \(L^2([a_j,a_j+T_j])\) with frame bound
\(1/c_j\).

We now translate both intervals to intervals beginning at the
origin. For any \(\widetilde{f}\in L^2([0,T_j])\), we let  $f(x)=\widetilde{f}(x-a_j)$ for 
   $x\in[a_j,a_j+T_j].$  
Then
\[
\int_{a_j}^{a_j+T_j}
f(x)e^{-2\pi i\xi_\lambda x}\,dx=
\int_0^{T_j}
\widetilde{f}(t)e^{-2\pi i\xi_\lambda(t+a_j)}\,dt=
e^{-2\pi i\xi_\lambda a_j}
\int_0^{T_j}
\widetilde{f}(t)e^{-2\pi i\xi_\lambda t}\,dt.\]
Therefore, \eqref{eq:parallel-frame-original} is equivalent to
\begin{equation}\label{eq:parallel-frame-zero}
    \sum_{\lambda\in\Lambda}
    \left|
        \int_0^{T_j}
        \widetilde{f}(t)e^{-2\pi i\xi_\lambda t}\,dt
    \right|^2
    =
    \frac{1}{c_j}
    \int_0^{T_j}|\widetilde{f}(t)|^2\,dt
\end{equation}
for every \(\widetilde{f}\in L^2([0,T_j])\).

Assume, without loss of generality, that
$T_1\leq T_2$.
Let \(g\in L^2([0,T_1])\) and 
\[
    g'(t)
    =
    \begin{cases}
        g(t),&0\leq t\leq T_1,\\
        0,&T_1<t\leq T_2.
    \end{cases}
\]
Applying \eqref{eq:parallel-frame-zero} with \(j=2\) to
\( g'\), we obtain
\[
\begin{aligned}
\sum_{\lambda\in\Lambda}
\left|
    \int_0^{T_2}
    g'(t)e^{-2\pi i\xi_\lambda t}\,dt
\right|^2
=
\frac{1}{c_2}
\int_0^{T_2}| g'(t)|^2\,dt.
\end{aligned}
\]
By the definition of \( g'\), this becomes
\begin{equation}\label{eq:bound-c2-on-short}
    \sum_{\lambda\in\Lambda}
    \left|
        \int_0^{T_1}
        g(t)e^{-2\pi i\xi_\lambda t}\,dt
    \right|^2
    =
    \frac{1}{c_2}
    \int_0^{T_1}|g(t)|^2\,dt.
\end{equation}
On the other hand, applying
\eqref{eq:parallel-frame-zero} directly with \(j=1\) gives
\begin{equation}\label{eq:bound-c1-on-short}
    \sum_{\lambda\in\Lambda}
    \left|
        \int_0^{T_1}
        g(t)e^{-2\pi i\xi_\lambda t}\,dt
    \right|^2
    =
    \frac{1}{c_1}
    \int_0^{T_1}|g(t)|^2\,dt.
\end{equation}
Comparing \eqref{eq:bound-c2-on-short} and
\eqref{eq:bound-c1-on-short}, we find that
\[
    \left(\frac{1}{c_1}-\frac{1}{c_2}\right)
    \int_0^{T_1}|g(t)|^2\,dt
    =
    0
\]
for every \(g\in L^2([0,T_1])\). Choosing any nonzero \(g\), we
conclude that $ c_1=c_2$.
Finally, the probability normalization gives $ c_1T_1+c_2T_2=1$.
Therefore, $c_1=c_2=\frac{1}{T_1+T_2}$.
This completes the proof.
\end{proof}
The same argument shows that a spectral probability measure supported on finitely many parallel line segments lying on distinct lines must have the same constant density on each segment. This yields the following corollary.
\begin{corollary}
 \label{coro:weighted-parallel-segments}
Given positive $k\in\mathbb{Z}$. Let \(h_1,h_2,\cdots,h_k\in\mathbb R\) satisfy  pairwise distinct,
\(a_1,a_2,\cdots,a_k\in\mathbb R\) and \(T_1,T_2,\cdots,T_k,c_1,c_2,\cdots,c_k>0\). And let
\[
    \mu
    =\sum_{j=1}^{k}
    c_j
    \left(
        \mathcal L\!\restriction_{[a_j,a_j+T_j]}
        \times\delta_{h_j}
    \right)
\]
is a Borel probability measure. If $\mu$ is a spectral measure, then $$c_1=c_2=\cdots=c_k=\frac{1}{T_1+T_2+\cdots+T_k}.$$  
\end{corollary}

\subsection{Weighted measures on two non-parallel coplanar
line segments}
In this section, we show that the spectrality problem for a weighted
measure supported on two non-parallel coplanar line segments can be reduced,
via an invertible linear transformation, to the corresponding
unweighted problem.

We first recall the standard invariance of spectral measures under
invertible linear transformations. Let \(\sigma\) be a finite Borel
measure on \(\mathbb R^d\), let
\(M:\mathbb R^d\to\mathbb R^d\) be an invertible affine transformation, and set
\[
    \nu=M_{\#}\sigma.
\]
Here, \(M_{\#}\sigma\) denotes the pushforward of \(\sigma\) under \(M\). More precisely, the measure \(\nu=M_{\#}\sigma\) is defined by
\[
    \nu(E)=\sigma\bigl(M^{-1}(E)\bigr)
\]for every Borel set \(E\subset\mathbb R^d\). Equivalently, for   every \(f\in L^1(\nu)\),
\[
    \int_{\mathbb R^d} f(y)\,d\nu(y)
    =
    \int_{\mathbb R^d} f(Mx)\,d\sigma(x).
\]Thus, \(\nu\) is the measure obtained by transporting \(\sigma\) through the  affine transformation \(M\).
By \cite[Lemma~2.3]{LiWang2024}, \(\sigma\) is spectral if and only if
\(\nu\) is spectral. More precisely,
\begin{equation}\label{eq:spectrum-linear-transformation}
    \Lambda \text{ is a spectrum for } \sigma
    \quad\Longleftrightarrow\quad
    M^{-T}\Lambda \text{ is a spectrum for } \nu.
\end{equation}
\begin{proposition}\label{prop:removing-weights}
Let \(I_1,I_2\subset\mathbb R^2\) be two non-parallel line segments,
and let
\[
    \mu
    =
    c_1\mathcal H^1\!\restriction_{I_1}
    +
    c_2\mathcal H^1\!\restriction_{I_2},
\]
with  $ c_1,c_2>0$. Then there exists an invertible linear transformation
\(A:\mathbb R^2\to\mathbb R^2\) such that
\begin{equation}\label{eq:unweighted-pushforward}
    A_{\#}\mu
    =
    \mathcal H^1\!\restriction_{A(I_1)}
    +
    \mathcal H^1\!\restriction_{A(I_2)}.
\end{equation}
\end{proposition}
\begin{proof}
Let \(v_1,v_2\) be unit direction vectors of \(I_1,I_2\),
respectively. Since \(I_1\) and \(I_2\) are non-parallel, the vectors
\(v_1,v_2\) are linearly independent and  form a basis of
\(\mathbb R^2\). Hence, there exists a unique linear transformation \(A:\mathbb R^2\to\mathbb R^2\) determined by
\begin{equation}\label{eq:def-A}
    Av_1=c_1v_1 
    \qquad \text{and} \qquad
    Av_2=c_2v_2.
\end{equation}
Since \(c_1,c_2>0\), the transformation \(A\) is invertible.

We claim that, for \(j=1,2\),
\begin{equation}\label{eq:individual-pushforward}
    A_{\#}
    \bigl(c_j\mathcal H^1\!\restriction_{I_j}\bigr)
    =
    \mathcal H^1\!\restriction_{A(I_j)}.
\end{equation}
Write
\[
    I_j=\{p_j+tv_j:0\leq t\leq T_j\},
\]
where  $T_j=\mathcal H^1(I_j)$.
For every \(f\in C_c(\mathbb R^2)\), we have
\[
\begin{aligned}
    \int_{\mathbb R^2} f\,d
    \left[
        A_{\#}
        \bigl(c_j\mathcal H^1\!\restriction_{I_j}\bigr)
    \right]
    &=
    c_j\int_{I_j}f(Ax)\,d\mathcal H^1(x)\\
    &=
    c_j\int_0^{T_j}f(Ap_j+tAv_j)\,dt\\
    &=
    c_j\int_0^{T_j}f(Ap_j+c_jtv_j)\,dt.
\end{aligned}
\]
Making the change of variables \(s=c_jt\), we obtain
\[
     \int_{\mathbb R^2} f\,d
    \left[
        A_{\#}
        \bigl(c_j\mathcal H^1\!\restriction_{I_j}\bigr)
    \right]
    =
    \int_0^{c_jT_j}f(Ap_j+sv_j)\,ds.
\]
Since
\[
    A(I_j)
    =
    \{Ap_j+sv_j:0\leq s\leq c_jT_j\},
\]
we have
\[
    \int_{\mathbb R^2} f\,d
    \left[
        A_{\#}
        \bigl(c_j\mathcal H^1\!\restriction_{I_j}\bigr)
    \right]= \int_{A(I_j)}f\,d\mathcal H^1.
\]
This proves \eqref{eq:individual-pushforward}. Summing the identities
for \(j=1,2\) gives  $$ A_{\#}\mu
    =
    \mathcal H^1\!\restriction_{A(I_1)}
    +
    \mathcal H^1\!\restriction_{A(I_2)}.$$
\end{proof}

\begin{remark}\label{rem:parallel-case}
The non-parallel assumption in Proposition~\ref{prop:removing-weights}
is essential. Suppose that \(I_1\) and \(I_2\) are parallel and have a
common unit direction vector \(v\). For every invertible linear
transformation \(A\),
\[
    A_{\#}
    \bigl(c_j\mathcal H^1\!\restriction_{I_j}\bigr)
    =
    \frac{c_j}{\lvert Av\rvert}
    \mathcal H^1\!\restriction_{A(I_j)},
    \qquad j=1,2.
\]
Hence the ratio of the two transformed densities remains unchanged:
\[
    \frac{c_1/\lvert Av\rvert}{c_2/\lvert Av\rvert}
    =
    \frac{c_1}{c_2}.
\]
Therefore, if \(c_1\neq c_2\), no invertible linear transformation can
simultaneously remove the two weights. This explains the fundamental
distinction between the parallel and non-parallel cases.
\end{remark}

\begin{proposition}[Reduction to the unweighted coordinate-axis model]
\label{cor:weighted-axis-reduction}
Let \(I_1,I_2\subset\mathbb R^2\) be two non-parallel line segments
with lengths
   $ T_j=\mathcal H^1(I_j)$ for $j=1,2$,
and let
\[
    \mu
    =
    c_1\mathcal H^1\!\restriction_{I_1}
    +
    c_2\mathcal H^1\!\restriction_{I_2},
    \qquad c_1,c_2>0.
\]
Then there exists an invertible affine transformation \(T: \mathbb R^2\to\mathbb R^2\) 
such that
\begin{equation}\label{eq:weighted-axis-model}
    T_{\#}\mu
    =
    \mathcal L\!\restriction_{[t_1,t_1+c_1T_1]}
    \times\delta_0
    +
    \delta_0\times
    \mathcal L\!\restriction_{[t_2,t_2+c_2T_2]}
\end{equation}
for some \(t_1,t_2\in\mathbb R\).
Consequently, \(\mu\) is spectral if and only if  \( T_{\#}\mu\) is spectral.
In particular,   every spectrum of \( T_{\#}\mu\) is contained in a straight line  if and only if  every spectrum of \( \mu\) is contained in a straight line.
\end{proposition}
\begin{proof}
By Proposition~\ref{prop:removing-weights}, there exists an
invertible linear transformation \(A\) such that
\[
    A_{\#}\mu
    =
    \mathcal H^1\!\restriction_{A(I_1)}
    +
    \mathcal H^1\!\restriction_{A(I_2)}.
\]
Moreover,
\[
    \mathcal H^1(A(I_j))=c_jT_j,
    \qquad j=1,2.
\]
Since \(A(I_1)\) and \(A(I_2)\) are non-parallel, it follows from
\cite[Remark~1(III)]{KSW} that there exists an invertible
length-preserving affine transformation \(S\) along the two segment
directions such that
\[
    S(A(I_1))
    \subset\mathbb R\times\{0\},
    \qquad
    S(A(I_2))
    \subset\{0\}\times\mathbb R.
\]
Consequently, for \(T=S\circ A\), we have
\[
    T_{\#}\mu
    =
    \mathcal L\!\restriction_{[t_1,t_1+c_1T_1]}
    \times\delta_0
    +
    \delta_0\times
    \mathcal L\!\restriction_{[t_2,t_2+c_2T_2]}
\]
for some \(t_1,t_2\in\mathbb R\).
The correspondence between the spectra follows from the invariance
of spectrality under invertible affine transformations.
\end{proof}
\begin{proof}[Proof of Theorem \ref{non-p}]  
The proof of the theorem can  be deduced from Proposition  \ref{cor:weighted-axis-reduction} and \cite[Theorem 2]{KSW}. 

\end{proof}
\section{Application}
We conclude with two examples illustrating the application of
our results to weighted measures on two non-parallel line segments.
The first shows how the ratio of the two densities determines
the possible directions of line spectra in the coordinate-axis
model. The second considers two segments with unequal densities
and explicitly constructs a spectrum of the corresponding measure.
\begin{example}[The weighted coordinate-axis model]
\label{ex:weighted-coordinate-axis}

Consider the weighted measure
\begin{equation}\label{eq:weighted-coordinate-measure}
    \mu
    =
    c_1
    \left(
        \mathcal L\!\restriction_{[t_1,t_1+T_1]}
        \times\delta_0
    \right)
    +
    c_2
    \left(
        \delta_0\times
        \mathcal L\!\restriction_{[t_2,t_2+T_2]}
    \right),
    \qquad c_1,c_2>0.
\end{equation}
Let
\[
    A=
    \begin{pmatrix}
        c_1&0\\
        0&c_2
    \end{pmatrix}.
\]
Since \(A\) stretches the \(x\)-axis and the \(y\)-axis by the
factors \(c_1\) and \(c_2\), respectively, a change of variables gives
\begin{equation}\label{eq:unweighted-coordinate-measure}
\begin{aligned}
    \nu:=A_{\#}\mu
    &=
    \mathcal L\!\restriction_
        {[c_1t_1,c_1(t_1+T_1)]}
        \times\delta_0+
    \delta_0\times
    \mathcal L\!\restriction_
        {[c_2t_2,c_2(t_2+T_2)]}.
\end{aligned}
\end{equation}
Thus, the weights \(c_1\) and \(c_2\) are absorbed into the lengths
and locations of the transformed segments. In particular, the
transformed segments have lengths \(c_1T_1\) and \(c_2T_2\),
respectively.

By the invariance of spectrality under invertible linear
transformations, \(\mu\) is a spectral measure with a spectrum $ \Lambda$ if and only if \(\nu\) is a spectral measure with a spectrum $A^{-T}\Lambda$. 

Suppose that \(\nu\) is spectral. By the characterization of
unweighted arc-length measures supported on two non-parallel line
segments in \cite[Theorems~2 and~3]{KSW}, every spectrum of \(\nu\)
containing the origin is contained in one of the two angle bisectors
\[
    L_+=\{(s,s):s\in\mathbb R\}
    \qquad\text{or}\qquad
    L_-=\{(s,-s):s\in\mathbb R\}.
\]
Since \(A^T=A\), the corresponding spectrum of \(\mu\) is contained
in either
\[
\begin{aligned}
    A^TL_+
    =
    \{(c_1s,c_2s):s\in\mathbb R\}  \qquad\text{or}\qquad
    A^TL_-
     =
    \{(c_1s,-c_2s):s\in\mathbb R\}.
\end{aligned}
\]
Equivalently, every spectrum of \(\mu\) containing the origin is
contained in one of the two lines
\begin{equation}\label{eq:weighted-spectral-lines}
    y=\frac{c_2}{c_1}x
    \qquad\text{or}\qquad
    y=-\frac{c_2}{c_1}x.
\end{equation}

Therefore, the ordinary angle bisectors \(y=x\) and \(y=-x\) in the
unweighted frequency model become the weighted angle bisectors in
\eqref{eq:weighted-spectral-lines}. In particular, the weights do
not produce a new type of spectrum. They only change the slopes of
the possible spectral lines through the   transformation \(A^T\).

If \(c_1=c_2\), the two weighted angle bisectors reduce to the
ordinary angle bisectors. If \(c_1\neq c_2\), their slopes become $
    \pm\frac{c_2}{c_1}.$ \end{example}
\begin{remark}\label{rem:weighted-slopes-projection}
  This dependence of the spectral directions on the weights also has
a natural interpretation in terms of the projection principle
developed in \cite[Theorem~1.2]{kolountzakis2025spectrality}. Let
\[
L_\theta=\{s(\cos\theta,\sin\theta):s\in\mathbb R\} 
\]
with $ \sin\theta\cos\theta\ne0$. Orthogonal projection onto \(L_\theta\) scales the lengths of the horizontal and vertical segments by \(|\cos\theta|\) and \(|\sin\theta|\), respectively. The corresponding projected components therefore have densities
\[
\frac{c_1}{|\cos\theta|}
\qquad\text{and}\qquad
\frac{c_2}{|\sin\theta|}.
\]
Since an absolutely continuous spectral measure must have constant
density on its support, these two densities must agree. Therefore,
\[
    \frac{c_1}{\lvert\cos\theta\rvert}
    =
    \frac{c_2}{\lvert\sin\theta\rvert},
\]
and consequently $ \lvert\tan\theta\rvert =\frac{c_2}{c_1}$.
Thus, from the projection viewpoint, the weights determine the
spectral direction because the projection must compensate for the
different densities on the two segments. This gives precisely the
two weighted angle bisectors obtained above from the 
transformation \(A^T\).
\end{remark}
 
 \begin{example} 
\label{ex:weighted-nonparallel}

Let
\[
    I_1=\{(s,0):0\leq s\leq 1\}
\]
and
\[
    I_2=
    \left\{
        \frac{s}{\sqrt{2}}(1,1):0\leq s\leq 1
    \right\}.
\]
Thus, \(I_1\) and \(I_2\) are two non-parallel line segments of
length \(1\). Consider the weighted probability measure
\begin{equation}\label{eq:weighted-example}
    \mu
    =
    \frac{1}{4}\mathcal H^1\!\restriction_{I_1}
    +
    \frac{3}{4}\mathcal H^1\!\restriction_{I_2}.
\end{equation}

Let
\[
    v_1=(1,0) 
    \qquad  \text{and}  \qquad
    v_2=\frac{1}{\sqrt{2}}(1,1)
\]
be unit direction vectors of \(I_1\) and \(I_2\), respectively.
Define the linear transformation \(A=
    \begin{pmatrix}
        \frac{1}{2}&-\frac{1}{2}\\[2mm]
        0&\frac{3\sqrt{2}}{2}
    \end{pmatrix} \) by
\[
    Av_1=(\frac{1}{2},0)
    \qquad \text{and}
 \qquad   Av_2=(0,\frac{3}{2}).
\]
It follows that
\[
    A(I_1)=\left[0,\frac{1}{2}\right]\times\{0\}
 \qquad \text{and} \qquad  A(I_2)=\{0\}\times\left[0,\frac{3}{2}\right].
\]
Since \(A\) stretches the tangential directions of \(I_1\) and
\(I_2\) by the factors \(1/2\) and \(3/2\), respectively, we obtain
\[
\begin{aligned}
    \nu:=A_{\#}\mu
     =
    \frac{1}{2}
    \mathcal L\!\restriction_{[0,1/2]}\times\delta_0 
     +
    \frac{1}{2}
    \delta_0\times
    \mathcal L\!\restriction_{[0,3/2]}.
\end{aligned}
\]
This is an unweighted coordinate-axis measure of the form considered
in \cite[Theorem~2]{KSW}, with
\[
    t_1=t_2=0,
    \qquad
    T_1=\frac{1}{2},
    \qquad
    T_2=\frac{3}{2}.
\]
Notice that
\[
    T_1+T_2=2,
    \qquad
    T_1\neq T_2,
    \qquad
    t_1+t_2=0\in2\mathbb Z.
\]
Therefore, \(\nu\) is a spectral measure. Moreover, by
\cite[Theorem~2(II)]{KSW}, its unique spectrum containing the origin
is
\[
    \Gamma
    =
    \left\{
        \left(\frac{n}{2},-\frac{n}{2}\right):
        n\in\mathbb Z
    \right\}.
\]

Since \(\nu=A_{\#}\mu\), the corresponding spectrum of the original
weighted measure \(\mu\) is
$ \Lambda=A^T\Gamma$.
A direct calculation gives
\[
\begin{aligned}
    A^T
    \begin{pmatrix}
        n/2\\[1mm]
        -n/2
    \end{pmatrix}
     =
    \begin{pmatrix}
        \frac{1}{2}&0\\[1mm]
        -\frac{1}{2}&\frac{3\sqrt{2}}{2}
    \end{pmatrix}
    \begin{pmatrix}
        n/2\\[1mm]
        -n/2
    \end{pmatrix} =
    \begin{pmatrix}
        n/4\\[1mm]
        -(1+3\sqrt{2})n/4
    \end{pmatrix}.
\end{aligned}
\]
Hence
\begin{equation}\label{eq:weighted-example-spectrum}
    \Lambda
    =
    \left\{
        \left(
            \frac{n}{4},
            -\frac{(1+3\sqrt{2})n}{4}
        \right):
        n\in\mathbb Z
    \right\}
\end{equation}
is a spectrum for \(\mu\).
This means \(\Lambda\) is contained in the straight line
 $y=-(1+3\sqrt{2})x.$
Furthermore, every spectrum of \(\mu\) is a translate of the set
in \eqref{eq:weighted-example-spectrum}.
\end{example}
\section*{Acknowledgments}
The author is grateful to Mihail N. Kolountzakis for reading the manuscript and for his valuable suggestions.


\bibliographystyle{alpha}

\begin{thebibliography}{KLM999}

\bibitem{ai2023spectrality}  W. H. Ai, Z. Y. Lu and T. Zhou.  The spectrality of symmetric additive measures. {\em Comptes Rendus. Math.}, 361: 783-793, 2023.

\bibitem {AHH19} L. X. An, L. He and X. G. He. Spectrality and non-spectrality of the Riesz product measures with three elements in digit sets. {\em J. Funct. Anal.}, 277: 255-278, 2019.

 \bibitem{AH14} L. X. An and X. G. He. A class of spectral Moran measures, {\em J. Funct. Anal.}, 266: 343-354, 2014.


\bibitem {AHL22} L. X. An, X. G. He and C. K. Lai. Classification of spectral self-similar measures with four-digit elements. {\em Asian J. Math.}, 27: 467-492, 2023.

\bibitem{DHL2013} X. R. Dai, X. G. He and C. K. Lai. Spectral property of Cantor measures with consecutive digits. {\em Adv. Math.},   242:  187-208, 2013.

\bibitem{DHL2014} X. R. Dai, X. G. He and K. S. Lau. On spectral N-Bernoulli measures. {\em Adv. Math.},  259:  511-531, 2014.

 
\bibitem {DL21} Q. R. Deng and M. T. Li. Spectrality of Moran-type self-similar measures on $\mathbb R$. {\em J. Math. Anal. Appl.},  506: 125547, 2022.

\bibitem {DL23} Q. R. Deng and M. T. Li, Spectrality of Moran-type Bernoulli convolutions. {\em Bull. Malays. Math. Sci. Soc.},   46: 136, 2023.



\bibitem{ducasse2025spectral} B. Ducasse, D. Dutkay and  C. Fernandez . Spectral properties of unions of intervals and groups of local translations. {\em ArXiv Preprint ArXiv:2506.18625}, 2025.

 \bibitem{DHL2019} D. Dutkay, J. Haussermann and C.K. Lai. Hadamard triples generate self-affine spectral measures. {\em  Trans. Am. Math. Soc.},   371:  1439-1481, 2019.
 
\bibitem{DL2014}D. E. Dutkay and C. K. Lai. Uniformity of measures with Fourier frames. {\em Adv.
Math.}, 252: 684-707, 2014.


\bibitem{FFS16} A. H. Fan, S. L. Fan and R. X. Shi. Compact open spectral sets in $ \mathbb{ Q}_p$. {\em J. Funct. Anal.},  271: 3628-3661, 2016.
\bibitem {FY97} D. J. Feng, Z. Y. Wen and J. Wu. Some dimensional results for homogeneous Moran sets. {\em Sci. China Math.},  40:  , 475-482, 1997.




\bibitem{farkas2006onfuglede} B. Farkas, M. Matolcsi and P. M\'{o}ra. On Fuglede's conjecture and the existence of universal spectra. {\em J. Fourier Anal. Appl.}, 12(5): 483-494, 2006.

\bibitem{farkas2006tiles} B. Farkas and Sz. Gy. R\'ev\'esz. Tiles with no spectra in dimension 4. {\em Math. Scand.},
98(1): 44-52, 2006.


 \bibitem{FHL15} X.Y. Fu, X.G. He and K.S. Lau. Spectrality of self-similar tiles. {\em Constr. Approx.},   42: 519-541, 2015.
  \bibitem{FHW18} Y.S. Fu, X.G. He and Z.Y. Wen. Spectra of B ernoulli convolutions and random convolutions. {\em J. Math. Pures Appl.},  116:105-131, 2018.
  
\bibitem{fuglede1974operators} B. Fuglede. Commuting self-adjoint partial differential operators and a group theoretic problem.  {\em J. Funct.  Anal.}, 16: 101-121, 1974.



\bibitem{JorgensenPedersen1998}
P. E. T. Jorgensen and S. Pedersen.
Dense analytic subspaces in fractal \(L^2\)-spaces.
{\em J. Anal. Math.}, 75:185-228, 1998.
 

\bibitem{kolountzakis2025non}  M. N. Kolountzakis and C. K. Lai.  Non-spectrality of some piecewise smooth curves and unions of line segments. {\em ArXiv Preprint ArXiv: 2507.00581},  2025.

 
\bibitem{kolountzakis2006tiles} M. N. Kolountzakis and M. Matolcsi. Tiles with no spectra. {\em Forum Math.}, 18:519-528, 2006

\bibitem{kolountzakis2006hadamard} M. N. Kolountzakis and M. Matolcsi. Complex Hadamard matrices and the spectral set conjecture. {\em Collect. Math.}, 57:281-291, 2006. 

\bibitem{KSW} M. N. Kolountzakis, R. X. Shi and S. Wu, Spectra for finite unions of line segments. ArXiv Preprint ArXiv:  2512.19872, 2025.
\bibitem{kolountzakis2025spectrality} M. N. Kolountzakis and  S. Wu. Spectrality of a measure consisting of two line segments. {\em J. Fourier Anal. Appl.}, 32(4), 2026.  


\bibitem{laba2001}  I. {\L}aba. Fuglede's conjecture for a union of two intervals. {\em Proc. Amer. Math. Soc.}, 129: 2965-2972, 2001.

\bibitem{lai2021spectral} C. K. Lai, B. C. Liu and H. Prince.  Spectral properties of some unions of linear spaces. {\em J. Funct. Anal.}, 280: 108985, 2021.

\bibitem{LiWang2024} W. X. Li and Z. Q. Wang. The spectrality of infinite convolutions in  $\mathbb{R}^d$. {\em  J. Fourier Anal. Appl.}, 30(3), 2024 

\bibitem{LevMatolcsi2022}
N. Lev and M. Matolcsi.
The Fuglede conjecture for convex domains is true in all dimensions.
{\em Acta Math.}, 228(2):385-420, 2022.

\bibitem{tao2004fuglede}
T. Tao. Fuglede’s conjecture is false in 5 and higher dimensions. {\em  Math. Res. Lett.}, 11(23): 251-258, 2004.

\bibitem{Zhang2026}
T. Zhang.
Both directions of Fuglede's conjecture fail in dimension two.
{\em ArXiv Preprint arXiv:2607.15632}, 2026.
 
\end{thebibliography}

\end{document}